\documentclass[11pt]{article}
\usepackage{amssymb}
\usepackage[top=4cm,bottom=2cm,left=2cm,right=2cm]{geometry}
\usepackage{graphicx}
\usepackage{amsmath,amsthm,amssymb}
\usepackage{enumerate}

\numberwithin{equation}{section}
\newtheorem{thm}{Theorem}[section]

\newtheorem{lem}[thm]{Lemma}
\newtheorem{prop}[thm]{Proposition}
\newtheorem{rem}[thm]{Remark}

\def\a{\alpha}
\def\b{\beta}

\def\d{\delta}
\def\D{\Delta}
\def\g{\gamma}
\def\G{\Gamma}

\def\na{\nabla}
\def\om{\omega}

\def\s{\sigma}

\def\th{\theta}

\def\ve{\varepsilon}

\def\wt{\widetilde}
\def\bb{\mathbb}

\def\cut{\setminus}

\DeclareMathOperator*\Res{{Res}} \DeclareMathOperator\A{\bf{A}}
\DeclareMathOperator\Ai{{Ai}} \DeclareMathOperator\Bi{{Bi}}
\DeclareMathOperator\re{{Re}} \DeclareMathOperator\im{{Im}}

\def\AND{\qquad\mbox{and}\qquad}

\def\C{\mathbb C}

\def\N{\mathbb N}

\begin{document}
\title{Global Asymptotics of the Meixner Polynomials}
\author{X.-S. Wang\thanks{Corresponding author. E-mail address: xswang4@mail.ustc.edu.cn}~ and R. Wong\\
Department of Mathematics\\ City University of Hong Kong\\ Tat
Chee Avenue, Kowloon, Hong Kong}
\date{}

\maketitle
\begin {center}
Dedicated to Professor Lee Lorch on his ninety-fifth birthday
\end {center}

\begin{abstract}
Using the steepest descent method for oscillatory Riemann-Hilbert problems introduced by Deift and Zhou [Ann. Math. {\bf 137}(1993), 295-368], we derive asymptotic formulas for the Meixner polynomials in two regions of the complex plane
separated by the boundary of a rectangle. The asymptotic formula on the boundary of the rectangle is obtained by taking limits from either inside or outside. Our results agree with the ones obtained earlier for $z$ on the positive real line by using the steepest descent method for integrals [Constr. Approx. {\bf 14}(1998), 113-150].
\end{abstract}
\vfill \noindent {\bf AMS Subject Classification:} Primary 41A60; Secondary 33C45.

\noindent \textbf{Keywords}: Global asymptotics; Meixner polynomials; Riemann-Hilbert problems; Airy function.
\newpage

\section{Introduction}
\indent
In this paper, we investigate the asymptotic behavior of the
Meixner polynomials. These
polynomials have many applications in statistical physics. For
instance, they are used in the study of the
shape fluctuations in a certain two dimensional random growth model;
see \cite{Jo00} and the references therein.

For $\b>0$ and $0<c<1$, the Meixner polynomials are given by
\begin{eqnarray}\label{Meixner}
  m_n(z;\b,c)=(\b)_n\cdot{}_2F_1(-n,-z;\b;1-c^{-1}).
\end{eqnarray}
They satisfy the discrete orthogonality relation
\begin{eqnarray}\label{orthogonality}
  \sum_{k=0}^\infty m_n(k;\b,c)m_p(k;\b,c)\frac{c^k(\b)_k}{k!}=(1-c)^{-\b}c^{-n}n!(\b)_n\d_{np}.
\end{eqnarray}
This notation is adopted in \cite[$\S$ 10.24]{EMOT53} and also in \cite{JW98}.

Using probabilistic arguments, Maejima and Van Assche \cite{MV85} have given an asymptotic formula
for $m_n(n\a;\b,c)$ when $\a<0$ and $\b$ is a positive integer.
Their result is in terms of elementary functions.
By using the steepest-descent method for integrals, Jin and Wong \cite{JW98} have derived two infinite asymptotic expansions for $m_n(n\a;\b,c)$;
one holds uniformly for $0<\ve\le \a\le 1+\ve$, and the other holds uniformly for $1-\ve\le \a\le M<\infty$.
Both expansions involve the parabolic cylinder function and its derivative.

In view of Gauss's contiguous relations for hypergeometric functions \cite[$\S$ 15.2]{AS70} and the connection formula [8] $$m_{n}(-x-\beta;\beta,c^{-1})=c^{n}m_{n}(x;\beta,c),$$
\noindent we may restrict our study to the case $1\le\b<2$ and $0<c<1$.
Fixing $0<c<1$ and $1\le\b<2$, we intend to investigate the large-$n$ behavior of
$m_n(nz-\b/2;\b,c)$ for $z$ in the whole complex plane, including neighborhood of the origin and regions extending to infinity. Our approach is based on the steepest-descent method for oscillatory Riemann-Hilbert problems,
first introduced by Deift and Zhou \cite{DZ93} for nonlinear partial differential equations, later developed in \cite{DKMVZ99} for orthogonal polynomials
with respect to exponential weights,
and further extended in \cite{BKMM03,BKMM07} to a general class of discrete orthogonal polynomials.

A direct application of the method in \cite{BKMM03,BKMM07} would, however, only give local asymptotics.
For instance, in the case of Meixner polynomials, one would have to divide the complex plane into at least six regions
(one near the origin, two near the two turning points and three in between, including an unbounded one), and
give correspondingly six different asymptotic formulas.
To reduce the number of these regions,
we shall make some modifications to the method in [2,$\,\,$3]. Our approach is motivated by the previous work in [4, 13, 15, 17, 18], and the main idea is to extend, as large as possible, the two regions of validity of the two asymptotic formulas near the two turning points. There have already been several examples in which we only need two regions with appropriate asymptotic formulas to cover the entire plane;
the Hermite polynomial [17] is one of such examples.
However, for discrete orthogonal polynomials, there might be cuts starting from the finite endpoints of the intervals of orthogonality. For instance, in the case of Krawtchouk polynomials considered in [4], there are two cuts $(-\infty,0]$ and $[1,\infty)$ where no asymptotic formulas are given. In the present paper, we shall give two asymptotic formulas for the
Meixner polynomial $m_{n}(nz-\beta/2;\beta,c)$, one valid inside a rectangle with two vertical boundary lines passimg through $z=0$ and $z=1$, and the other valid outside the rectangle. Both formulas can be extended slightly beyond the boundary of the rectangle, and they are asymptotically equal to each other in the overlapping region.
The material in this paper is arranged as follows.
In Section 2, we use a standard method to
relate the Meixner polynomials to a Riemann-Hilbert problem for a matrix-valued function.
The motivation and details of this standard procedure can be found in \cite{BKMM03,BKMM07} and the reference given there.
In Section 3, we introduce some auxiliary functions which will be used in Section 4 for the construction of our parametrix.
In Section 4, we also prove that this parametrix is asymptotically equal to the solution of the Riemann-Hilbert problem formulated in Section 2.
In Section 5, we state our main result and make the
remark that our formulas agree with the ones already existing in the literature.

\section{Standard formulation of Riemann-Hilbert problem}
From (\ref{Meixner}), we note that the leading coefficient of
$m_n(z;\b,c)$ is $(1-c^{-1})^n$. Thus, the monic Meixner polynomials
are given by
\begin{eqnarray}\label{monic Meixner}
  \pi_n(z):=(1-\frac 1 c)^{-n}m_n(z;\b,c).
\end{eqnarray}
For convenience, in (2.1) we have suppressed the dependence of
$\pi_n(z)$ on $c$ and $\b$. Furthermore, throughout the paper we
shall fix the parameters $c\in(0,1)$ and $\b\in[1,2)$. The
orthogonality property of $\pi_n(z)$ can be easily derived from
(\ref{orthogonality}), and we have
\begin{eqnarray}\label{monic orthogonality}
  \sum_{k=0}^\infty \pi_n(k)\pi_p(k)w(k)=\d_{np}/\g_n^2,
\end{eqnarray}
where
\begin{eqnarray}\label{gamma}
  \g_n^2=\frac{(1-c)^{2n+\b}c^{-n}}{\G(n+\b)\G(n+1)}
\end{eqnarray}
and
\begin{eqnarray}\label{weight}
  w(z):=\frac{\G(z+\b)}{\G(z+1)}c^z.
\end{eqnarray}
Let $P(z)$ be the $2\times2$ matrix defined by
\begin{eqnarray}\label{P}
  P(z)=\left(\begin{matrix}
    P_{11}(z)& P_{12}(z)\\
    \\
    P_{21}(z)& P_{22}(z)
  \end{matrix}\right):=\left(\begin{matrix}
    \pi_n(z)&\sum\limits_{k=0}^\infty\cfrac{\pi_n(k)w(k)}{z-k}\\
    \\
    \g_{n-1}^2\pi_{n-1}(z)&\sum\limits_{k=0}^\infty\cfrac{\g_{n-1}^2\pi_{n-1}(k)w(k)}{z-k}
  \end{matrix}\right).
\end{eqnarray}
A proof of the following result can be found in \cite[Section
1.5.1]{BKMM07}. The only difference is that their $N$ should be
replaced by $\infty$.\\

\begin{prop}\label{prop-P}
The matrix-valued function $P(z)$ defined in (\ref{P}) is the unique solution of the
following interpolation problem:
\begin{enumerate}[(P1)]
\item  $P(z)$ is analytic in
  $\bb{C}\setminus\bb{N}$;
\item at each $z=k\in\bb{N}$, the first
column of
  $P(z)$ is analytic and the second column of $P(z)$ has a
  simple pole with residue
  \begin{eqnarray}\label{P2}
  \Res_{z=k}P(z)=\lim_{z\to k}P(z)\left(\begin{matrix}
    0&w(z)\\
    0&0
  \end{matrix}\right)=\left(\begin{matrix}
    0\ &w(k)P_{11}(k)
    \\
    0&w(k)P_{21}(k)
  \end{matrix}\right);
  \end{eqnarray}
\item for $z$ bounded away from $\bb{N}$,  $P(z)\left(\begin{matrix}
    z^{-n}&0\\
    0&z^n
  \end{matrix}\right)=I+O(|z|^{-1})$ as $z\to\infty$.
\end{enumerate}
\end{prop}

Let $\bb{X}$ denote the set defined by
\begin{eqnarray}\label{X}
  \bb{X}:=\{X_k\}_{k=0}^\infty,\qquad \mbox{where}\qquad
  X_k:=\frac{k+\b/2}n;
\end{eqnarray}
cf. [3], [4] and [13]. The $X_k$'s are called {\it nodes}. For the
sake of simplicity, we put
\begin{eqnarray}
  B(z):=\prod\limits_{j=0}^{n-1}(z-X_j).\label{B}
\end{eqnarray}
Our first transformation is given by
\begin{eqnarray}\label{Q}
  Q(z)&:=&n^{-n\s_3}P(nz-\b/2)B(z)^{-\s_3}\nonumber
  \\&=&\left(\begin{matrix}
    n^{-n}&0\\
    0&n^n
  \end{matrix}\right)P(nz-\b/2)\left(\begin{matrix}
    B(z)^{-1}&0\\
    0&B(z)
  \end{matrix}\right),
\end{eqnarray}
where $\s_3:=\left(\begin{matrix}
    1&0\\
    0&-1
  \end{matrix}\right)$ is a Pauli matrix.
In this paper, we shall also make use of another Pauli matrix,
namely $\s_1:=\left(\begin{matrix}
    0&1\\
    1&0
  \end{matrix}\right)$; see Section 4.

\begin{prop}\label{prop-Q}
The matrix-valued function $Q(z)$ defined in (\ref{Q}) is the unique solution of the
following interpolation problem:
\begin{enumerate}[(Q1)]
\item $Q(z)$ is analytic in
  $\mathbb{C}\setminus \mathbb{X}$;
\item at each node $X_k$ with $k\in\mathbb{N}$ and $k\ge n$, the
  first column of $Q(z)$ is analytic and the second column of $Q(z)$ has a
  simple pole with residue
  \begin{eqnarray}\label{Q2a}
  \Res_{z=X_k}Q(z)=\lim_{z\to X_k}Q(z)\left(\begin{matrix}
    0&w(nz-\b/2)B(z)^2\\
    0&0
  \end{matrix}\right);
  \end{eqnarray}
  at each node $X_k$ with $k\in\mathbb{N}$ and $k< n$, the
  second column of $Q(z)$ is analytic and the first column of $Q(z)$ has a
  simple pole with residue
  \begin{eqnarray}\label{Q2b}
  \Res_{z=X_k}Q(z)=\lim_{z\to X_k}Q(z)\left(\begin{matrix}
    0&0\\
    \\
    \cfrac{(z-X_k)^2}{w(nz-\b/2)B(z)^2}&0
  \end{matrix}\right);
  \end{eqnarray}
\item for $z$ bounded away from $\bb{X}$, $Q(z)=I+O(|z|^{-1})$ as $z\to\infty$.
\end{enumerate}
\end{prop}

{\it Proof.} This is obvious from Proposition \ref{prop-P} and the
definition of $Q(z)$ in (\ref{Q}). $\blacksquare$\\

The purpose of our next transformation is to remove the poles in the
interpolation problem for $Q(z)$ (cf. \cite[Section 4.2]{BKMM07}).
Let $\d>0$ be a sufficiently small number. We define (see Figure
\ref{fig-Q2R} below)
\begin{subequations}\label{R}
  \begin{equation}
    R(z):=Q(z)\left(\begin{matrix}
    1&0\\
    -\D_\pm(z)&1
  \end{matrix}\right)
  \end{equation}
  for $\re z\in(0,1)$ and $\pm\im z\in(0,\d)$, and
  \begin{equation}
    R(z):=Q(z)\left(\begin{matrix}
    1&-\na_\pm(z)\\
    0&1
  \end{matrix}\right)
  \end{equation}
  for $\re z\in(1,\infty)$ and $\pm\im z\in(0,\d)$, and
  \begin{equation}
    R(z):=Q(z)
  \end{equation}
  for $\re z\notin[0,\infty)$ or $\im z\notin[-\d,\d]$,
\end{subequations}
where
\begin{eqnarray}
\na_\pm(z):=\cfrac{n\pi w(nz-\b/2)B(z)^2}{e^{\mp i\pi (nz-\b/2)}\sin(n\pi z-\b\pi/2)},\label{nabla}\\
\D_\pm(z):=\cfrac{e^{\pm i\pi(nz-\b/2)}\sin(n\pi z-\b\pi/2)}{n\pi w(nz-\b/2)B(z)^2}.\label{Delta}
\end{eqnarray}

\begin{lem}\label{lem-R}
  For each $k\in\bb{N}$, the singularity of $R(z)$ at the node $X_k=\frac{k+\b/2}{n}$ is removable, that is,
  $\Res\limits_{z=X_k}R(z)=0.$
\end{lem}

\begin{proof}
For any $k\in\N$ with $k\ge n$, we have $X_k=\frac{k+\b/2}n>1$ since $1\le\b<2$. From (2.13), it is evident that the residue of $\na_\pm(z)$ at $z=X_k$ is
$$\Res_{z=X_k}\na_\pm(z)=w(nX_k-\b/2)B(X_k)^2.$$
From (\ref{Q2a}), we also note that the residue of $Q_{12}(z)$ at $z=X_k$ is $Q_{11}(X_k)$ multiplied by
$w(nX_k-\b/2)B(X_k)^2$.
Thus, it follows from (2.12b) that the residue of $R_{12}(z)=Q_{12}(z)-\na_\pm(z)Q_{11}(z)$ at $z=X_k$ is zero.
Similarly, one can show that
$\Res\limits_{z=X_k}R_{22}(z)=0.$ Since $R_{11}(z)= Q_{11}(z)$ and $R_{21}(z)= Q_{21}(z)$, and since $Q_{11}(z)$ and $Q_{21}(z)$ are analytic by Proposition 2.2, the residues of $R_{11}(z)$ and $R_{21}(z)$ at $X_{k}$ are zero.
For any $k\in\N$ with $k<n$, we have $X_k=\frac{k+\b/2}n<1$ since $1\le\b<2$. From (\ref{B}), (\ref{Q2b}) and (\ref{Delta}), we observe that
\begin{eqnarray*}
  \Res_{z=X_k}\D_\pm(z)=\cfrac 1{w(nX_k-\b/2)}\prod\limits_{\substack{j=0\\j\ne k}}^{n-1}(X_k-X_j)^{-2},\\
  \Res_{z=X_k}Q_{11}(z)=\cfrac {Q_{12}(X_k)}{w(nX_k-\b/2)}\prod\limits_{\substack{j=0\\j\ne k}}^{n-1}(X_k-X_j)^{-2}.
\end{eqnarray*}
Thus, the residue of $R_{11}(z)=Q_{11}(z)-\D_\pm(z)Q_{12}(z)$ at $z=X_k$ is zero.
Similarly, one can prove that the residue of $R_{21}(z)$ at $z=X_k$ is also zero. Since $R_{12}(z)=Q_{12}(z)$ and $R_{22}(z)$=$Q_{22}(z)$, and since $Q_{12}(z)$ and $Q_{22}(z)$ are analytic by Proposition 2.2, the residues of $R_{12}(z)$ and $R_{22}(z)$ at $X_{k}$ are zero. This completes the proof of the lemma. $\blacksquare$
\end{proof}
From the definition in (\ref{R}) and Lemma \ref{lem-R}, the jump conditions of $R(z)$ given in the follow proposition are easily verified.

\begin{prop}\label{prop-JoR}
Let $\Sigma_R$ be the oriented contour shown in Figure \ref{fig-Q2R}.
Denote by $R_+(z)$ and $R_{-}(z)$, respectively, the limiting values of $R(z)$ on $\Sigma_R$ taken from the left and from the right of the contour. The jump matrix $J_R(z):=R_-(z)^{-1}R_+(z)$ has the following explicit expressions.
For ${\textup Re}\,z=1$ and ${\textup I\textup m}\,z\in(-\delta,\delta)$, we have
\begin{eqnarray}\label{JoR1}
  J_R(z)=\left(\begin{matrix}
    1-\D_\pm(z)\na_\pm(z)&\na_\pm(z)\\
    -\D_\pm(z)&1
  \end{matrix}\right).
\end{eqnarray}
On the positive real line, we have
\begin{subequations}\label{JoR2}
  \begin{equation}
    J_R(x)=\left(\begin{matrix}
    1&0\\
    \D_-(x)-\D_+(x)&1
  \end{matrix}\right)
  \end{equation}
  for $x\in(0,1)$, and
  \begin{equation}
    J_R(x)=\left(\begin{matrix}
    1&\na_-(x)-\na_+(x)\\
    0&1
  \end{matrix}\right)
  \end{equation}
  for $x\in(1,\infty)$.
\end{subequations}
Furthermore, we have
\begin{subequations}\label{JoR3}
  \begin{equation}
    J_R(z)=\left(\begin{matrix}
    1&0\\
    \D_\pm(z)&1
  \end{matrix}\right)
  \end{equation}
  for $z=i\,\im\,z$ with $\im\,z \in(-\delta, \delta)$ and $z=\re\,z \pm i \delta$ with $\re\,z \in (0,1)$, and
  \begin{equation}
    J_R(z)=\left(\begin{matrix}
    1&\na_\pm(z)\\
    0&1
  \end{matrix}\right)
  \end{equation}
  for $z=\re z\pm i\d$ with $\re z\in(1,\infty)$.
\end{subequations}
\end{prop}

%\begin{proof}
%  Trivial.
%\end{proof}
%\newpage00000
\begin{figure}[htp]
\centering
\includegraphics{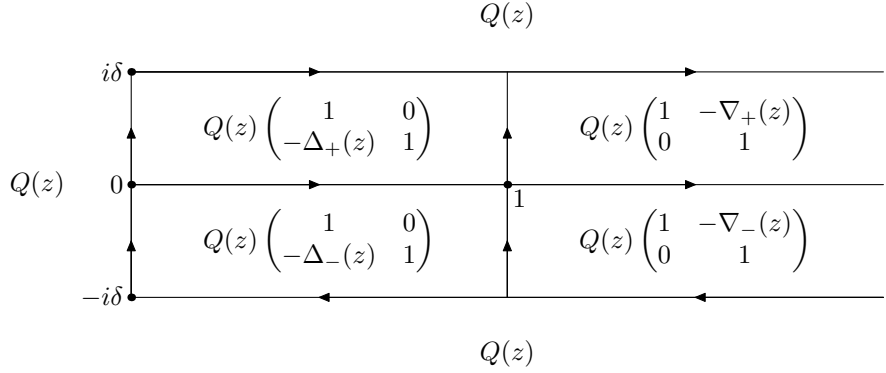}
\caption{The transformation $Q\to R$ and the oriented contour $\Sigma_R$.}\label{fig-Q2R}
\end{figure}

For simplicity, we define
\begin{eqnarray}
  \th(z)&:=&n\pi z-\b\pi/2,\label{theta}\\
  v(z)&:=&-z\log c,\label{v}\\
  C&:=&2i\pi c^{-\b/2}n^\b,\label{C}\\
  W(z)&:=&\frac{(nz)^{1-\b}\G(nz+\b/2)}{\G(nz+1-\b/2)}.\label{W}
  %W(z)&:=&2in\pi w(nz-\b/2)e^{nv(z)}=\frac{2in\pi \G(nz+\b/2)c^{-\b/2}}{\G(nz+1-\b/2)}.\label{W}
\end{eqnarray}
In view of (\ref{weight}) and the above notations, the functions defined in (\ref{nabla}) and (\ref{Delta}) become
\begin{eqnarray}
  \na_\pm=\frac{Cz^{\b-1}WB^2}{2i\sin\th e^{\mp i\th+nv}}\AND
  \D_\pm=\frac{2i\sin\th e^{\pm i\th+nv}}{Cz^{\b-1}WB^2}.\label{naD}
\end{eqnarray}
It is easy to see that $$\D_\pm\na_\pm=e^{\pm 2i\th}$$ for $z\in\C_\pm$. Also,
$$\na_--\na_+=-Cx^{\b-1}e^{-nv}WB^2$$
for $z=x\in(1,\infty)$, and
$$\D_--\D_+=\frac{4\sin^2\th}{Cx^{\b-1}e^{-nv}WB^2}$$
for $z=x\in(0,1)$.
\begin{prop}\label{prop-R}
The matrix-valued function $R(z)$ defined in (\ref{R}) is the unique solution of the
following Riemann-Hilbert problem:
\begin{enumerate}[(R1)]
\item $R(z)$ is analytic in
  $\mathbb{C}\setminus \Sigma_R$;
\item for $z\in\Sigma_R$,
  $
    R_+(z)=R_-(z)J_R(z),
  $ where the jump matrix $J_R(z)$ is given in Proposition \ref{prop-JoR};
\item for $z\in\mathbb{C}\setminus \Sigma_R$, $R(z)=I+O(|z|^{-1})$ as $z\to\infty$.
\end{enumerate}
\end{prop}

{\it Proof.}  It follows readily from Proposition \ref{prop-Q} and the definition of $R(z)$ in (\ref{R}).\qquad \qquad $\blacksquare$

%\newpage00000
\section{Some auxiliary functions}
To construct our parametrix, we should introduce some auxiliary functions.
First, define the two constants
\begin{eqnarray}\label{ab}
  a:=\frac{1-\sqrt{c}}{1+\sqrt{c}}\AND
  b:=\frac{1+\sqrt{c}}{1-\sqrt{c}}.
\end{eqnarray}
These constants are the two turning points for the Meixner polynomials; see \cite[(2.6)]{JW98}.
\begin{subequations}\label{phi}
Let
\begin{eqnarray}
\phi(z):=z\log\frac{\sqrt{bz-1}+\sqrt{az-1}}{\sqrt{bz-1}-\sqrt{az-1}}-\log\frac{\sqrt{z-a}+\sqrt{z-b}}{\sqrt{z-a}-\sqrt{z-b}}
\end{eqnarray}
for $z\in\C\cut(-\infty,b]$ and
\begin{eqnarray}
\wt\phi(z):=z\log\frac{\sqrt{1-az}+\sqrt{1-bz}}{\sqrt{1-az}-\sqrt{1-bz}}-\log\frac{\sqrt{b-z}+\sqrt{a-z}}{\sqrt{b-z}-\sqrt{a-z}}
\end{eqnarray}
for $z\in\C\cut(-\infty,0]\cup[a,\infty)$.
\end{subequations}
These two functions are analogues of the $\phi$-function and $\wt\phi$-function in \cite{WZ06}.
It is clear from the defintions that
\begin{eqnarray}\label{phi-phi'}
\wt\phi(z)=\phi(z)\pm i\pi(1-z)
\end{eqnarray}
for $z\in\C_\pm$.
As $z\to\infty$, we have
\begin{eqnarray*}
\phi(z)= z\log\frac{\sqrt b+\sqrt a}{\sqrt b-\sqrt a}-\log z+\log\frac{b-a}4-1+O(\frac{1}{z}).
\end{eqnarray*}
Here we have used the fact that $ab=1$. Put
\begin{eqnarray}
l:=2\log\frac{b-a}4-2,\label{l}
\end{eqnarray}
and recall the definition of $v(z)$ in (\ref{v}). Since $(\sqrt{b}+\sqrt{a})/(\sqrt{b}-\sqrt{a})=1/\sqrt{c}$ by (\ref{ab}),
it follows from the above two equations that
\begin{eqnarray}\label{phi-asymp}
-\phi(z)+v(z)/2+l/2=\log z+O(\frac{1}{z})
\end{eqnarray}
as $z\to\infty$.
For convenience, we define
\begin{eqnarray}
F(z):=\left[\frac32n\phi(z)\right]^{2/3}\AND \wt F(z):=\left[-\frac32n\wt\phi(z)\right]^{2/3}.\label{F}
\end{eqnarray}
Note by (\ref{phi}) that $\phi(b)=0$ and
\begin{eqnarray}\label{nu}
\phi'(z)=\log\frac{\sqrt{bz-1}+\sqrt{az-1}}{\sqrt{bz-1}-\sqrt{az-1}}= \log\frac{\sqrt{1-bz}+\sqrt{1-az}}{\sqrt{1-bz}-\sqrt{1-az}} \quad.
\end{eqnarray}
Using (3.3), it is readily seen that $\tilde{\phi}(a)=0$ and $\tilde{\phi}(0)=\frac{1}{2}\log c.$  The mapping properties of the functions $\phi(z)$ and $\wt\phi(z)$ are illustrated in Figure \ref{fig-phi}.
From this figure and the definitions (\ref{phi}) and (\ref{F}), we have the following proposition.

%\newpage00000
\begin{prop}\label{prop-phi}
For $z\in\C\cut[a,b]$, we have
\begin{eqnarray}\label{phi1}
\pm\arg F(z)\in(-\pi,\pi) \AND \arg\wt F(z)\in(-\pi,\pi).
\end{eqnarray}
For $\re z\in(a,b)$ and $\pm\im z\in[0,\d]$, we have
\begin{eqnarray}\label{phi2}
\pm\arg F(z)\in(\pi/3,\pi] \AND \qquad \mp\arg\wt F(z)\in(\pi/3,\pi].
\end{eqnarray}
For $x\ge0$ \it{and $\delta$ sufficient small, we have }
\begin{eqnarray}\label{phi3}
\re\phi(x\pm i\d)\sim\begin{cases}
\phi(x),&x\ge b;\\
-2\d\arctan\sqrt{\cfrac{1-ax}{bx-1}},&a<x<b;\\
\wt\phi(x)-\pi\d,&0\le x\le a.
\end{cases}
\end{eqnarray}
\end{prop}

{\it Proof.} It is easy to prove (\ref{phi1}) and (\ref{phi2}) by using (\ref{phi}), (\ref{F}) and Figure \ref{fig-phi}.
For small $\d>0$, we have from a two-term Taylor expansion
$$\re\phi(x\pm i\d)\sim\re\phi_\pm(x)\mp\d\im\phi'_\pm(x).$$
Note by (\ref{nu}) that
\begin{eqnarray*}
\im\phi'_\pm(x)=\begin{cases}
0,&x\ge b;\\
\pm2\arctan\sqrt{\cfrac{1-ax}{bx-1}},&a<x<b;\\
\pm\pi,&0\le x\le a.
\end{cases}
\end{eqnarray*}
Moveover, (\ref{phi}) and (\ref{phi-phi'}) imply
\begin{eqnarray*}
\re\phi_\pm(x)=\begin{cases}
\phi(x),&x\ge b;\\
0,&a<x<b;\\
\wt\phi(x),&0\le x\le a.
\end{cases}
\end{eqnarray*}
Thus, (\ref{phi3}) follows from the above three equations. \qquad\qquad $\blacksquare$\\

\begin{figure}[htp]
\centering
\includegraphics{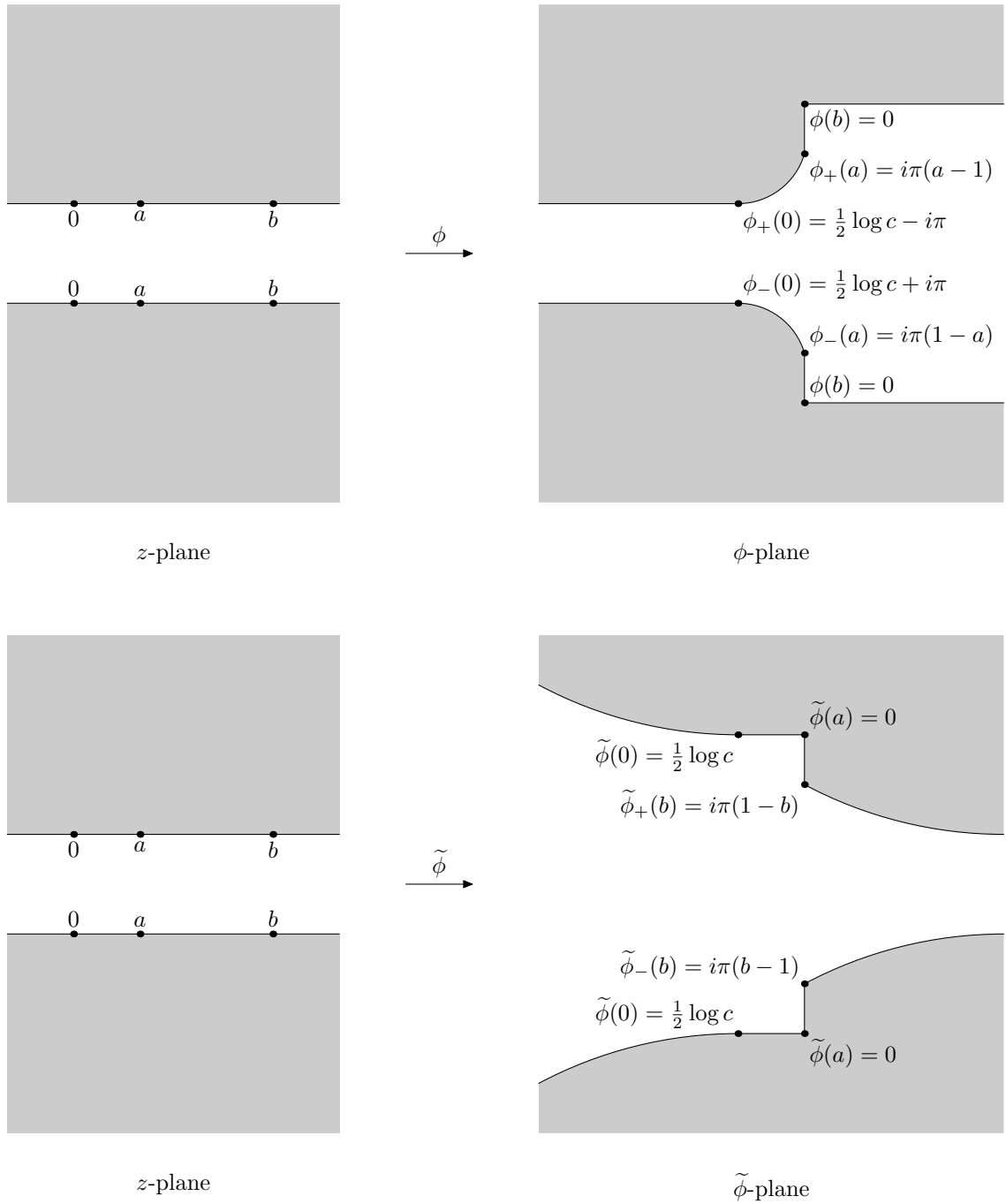}
\caption{The $z$-plane under the mappings $\phi(z)$ and $\wt\phi(z)$.}\label{fig-phi}
\end{figure}

%\newpage00000

Define
\begin{eqnarray}\label{N}
N(z):=\left(\begin{matrix}
    \cfrac{(z-1)^{\frac{1-\b}2}(\frac{\sqrt{z-a}+\sqrt{z-b}}2)^\b}{(z-a)^{1/4}(z-b)^{1/4}}&
    \cfrac{-i(z-1)^{\frac{\b-1}2} (\frac{\sqrt{z-a}-\sqrt{z-b}}2)^\b}{(z-a)^{1/4}(z-b)^{1/4}} \\
    \\
    \cfrac{i(z-1)^{\frac{1-\b}2}(\frac{\sqrt{z-a}-\sqrt{z-b}}2)^{2-\b}}{(z-a)^{1/4}(z-b)^{1/4}}&
    \cfrac{(z-1)^{\frac{\b-1}2} (\frac{\sqrt{z-a}+\sqrt{z-b}}2)^{2-\b}}{(z-a)^{1/4}(z-b)^{1/4}}
  \end{matrix}\right).
\end{eqnarray}
It is easy to verify that $N(z)$ is analytic in $\C\cut[a,b]$ and
\begin{eqnarray}\label{JoN}
N_+(x)=N_-(x)\left(\begin{matrix}
    0&-|x-1|^{\b-1} \\
    |x-1|^{1-\b}&0
  \end{matrix}\right),\qquad x\in(a,b).
\end{eqnarray}
The matrix $N(z)$ is analogous to the matrix $N(z)$ in \cite{DKMVZ99,WW05}.
Now we introduce the Airy parametrix which is also similar to the one in \cite{DKMVZ99,WW05}.
For $z\in\C_\pm$, define
\begin{eqnarray}\label{A}
  \A(z):= \left(\begin{matrix}
    \Ai(z)&-i\Bi(z) \\
    i\Ai'(z)&\Bi'(z)
  \end{matrix}\right) \left(\begin{matrix}
    1&\mp1/2 \\
    0&1/2
  \end{matrix}\right).
\end{eqnarray}
It is clear that
\begin{eqnarray}\label{JoA}
  \A_+(x)=\A_-(x)\left(\begin{matrix}
    1&-1 \\
    0&1
  \end{matrix}\right)
\end{eqnarray}
on the real line.
For convenience, set $\om=e^{2\pi i/3}$. Note that (cf. \cite[(10.4.9)]{AS70})
$$2\om\Ai(\om z)=-\Ai(z)+i\Bi(z)\AND 2\om^2\Ai(\om^2z)=-\Ai(z)-i\Bi(z).$$
We obtain from (\ref{A})
\begin{eqnarray}\label{A1}
\A(z)=\begin{cases}
  \left(\begin{matrix}
    \Ai(z)&\om^2\Ai(\om^2z) \\
    i\Ai'(z)&i\om\Ai'(\om^2z)
  \end{matrix}\right)&z\in\C_+;\\
  \\
  \left(\begin{matrix}
    \Ai(z)&-\om\Ai(\om z) \\
    i\Ai'(z)&-i\om^2\Ai'(\om z)
  \end{matrix}\right)&z\in\C_-.
\end{cases}
\end{eqnarray}
Furthermore, in view of (cf. \cite[(10.4.7)]{AS70})
$$\Ai(z)+\om\Ai(\om z)+\om^2\Ai(\om^2z)=0,$$
we have
\begin{eqnarray}\label{A2}
\A(z)\left(\begin{matrix}
    1&0 \\
    \pm1&1
  \end{matrix}\right)
=\begin{cases}
  \left(\begin{matrix}
    -\om\Ai(\om z)&\om^2\Ai(\om^2z) \\
    -i\om^2\Ai'(\om z)&i\om\Ai'(\om^2z)
  \end{matrix}\right)&z\in\C_+;\\
  \\
  \left(\begin{matrix}
    -\om^2\Ai(\om^2z)&-\om\Ai(\om z) \\
    -i\om\Ai'(\om^2z)&-i\om^2\Ai'(\om z)
  \end{matrix}\right)&z\in\C_-.
\end{cases}
\end{eqnarray}
Recall the asymptotic expansions of the Airy function and its
derivative (cf. \cite[p. 392]{Ol97} or \cite[p. 47]{Wo89})
\begin{eqnarray}\label{Airy}
 \Ai(z)\sim \frac{z^{-1/4}}{2\sqrt\pi}e^{-\frac 2 3z^{3/2}}\sum_{s=0}^\infty
 \frac{(-1)^su_s}{(\frac 2 3z^{3/2})^s},\indent
 \Ai'(z)\sim -\frac{z^{1/4}}{2\sqrt\pi}e^{-\frac 2 3z^{3/2}}\sum_{s=0}^\infty
 \frac{(-1)^sv_s}{(\frac 2 3z^{3/2})^s}
\end{eqnarray}
as $z\to\infty$ with $|\arg z|<\pi$, where $u_s, v_s$ are constants with
$u_0=v_0=1$. For $\arg z\in(-\pi,0]$, we have $\arg(\om z)\in(-\pi/3,2\pi/3]$.
Thus, by using (\ref{Airy}) we obtain as $z\to\infty$ with $\arg z\in(-\pi,0]$,
\begin{eqnarray*}
&&-\om\Ai(\om z)\sim\frac{-\om(\om z)^{-1/4}}{2\sqrt\pi}e^{-\frac23(\om z)^{3/2}}
\sim\frac{-iz^{-1/4}}{2\sqrt\pi}e^{\frac23 z^{3/2}},\\
&&-i\om^2\Ai'(\om z)\sim\frac{i\om^2(\om z)^{1/4}}{2\sqrt\pi}e^{-\frac23(\om z)^{3/2}}
\sim\frac{z^{1/4}}{2\sqrt\pi}e^{\frac23 z^{3/2}}.
\end{eqnarray*}
For $\arg z\in[0,\pi)$, we have $\arg(\om^2z)\in[4\pi/3,7\pi/3)$.
Here, we cannot use (\ref{Airy}) with $z$ replaced by $\om^2z$.
However, since $\om^2z=\om^{-1}z$ and $\arg(\om^{-1}z)\in[-2\pi/3,\pi/3)$,
we can use (\ref{Airy}) with $z$ replaced by $\om^{-1}z$ and obtain, as $z\to\infty$ with $\arg z\in[0,\pi)$,
\begin{eqnarray*}
&&\om^2\Ai(\om^2z)=\om^2\Ai(\om^{-1}z)\sim\frac{\om^2(\om^{-1} z)^{-1/4}}{2\sqrt\pi}e^{-\frac23(\om^{-1} z)^{3/2}}
\sim\frac{-iz^{-1/4}}{2\sqrt\pi}e^{\frac23 z^{3/2}},\\
&&i\om\Ai'(\om^2 z)=i\om\Ai'(\om^{-1} z)\sim\frac{-i\om(\om^{-1} z)^{1/4}}{2\sqrt\pi}e^{-\frac23(\om^{-1} z)^{3/2}}
\sim\frac{z^{1/4}}{2\sqrt\pi}e^{\frac23 z^{3/2}}.
\end{eqnarray*}
Applying (\ref{Airy}) and the above four formulas to (\ref{A1}) gives
\begin{eqnarray}\label{A-asymp1}
   \A(z)=\frac {z^{-\s_3/4}}{2\sqrt\pi} \left(\begin{matrix}
    1&-i \\
    -i&1
  \end{matrix}\right)(I+O(|z|^{-3/2}))e^{-\frac 2 3z^{3/2}\s_3}
\end{eqnarray}
as $z\to\infty$ with $\arg z\in(-\pi,\pi)$.
For $\arg z\in(\pi/3,\pi]$, we have $\arg(w^{-2}z)\in(-\pi,-\pi/3]$ and $\arg(w^{-1}z)\in(-\pi/3,\pi/3]$.
Thus, by using (\ref{Airy}) we obtain as $z\to\infty$ with $\arg z\in(\pi/3,\pi]$
\begin{eqnarray*}
&&-\om\Ai(\om z)=-\om\Ai(\om^{-2}z)\sim\frac{-\om(\om^{-2} z)^{-1/4}}{2\sqrt\pi}e^{-\frac23(\om^{-2} z)^{3/2}}
\sim\frac{z^{-1/4}}{2\sqrt\pi}e^{-\frac23 z^{3/2}},\\
&&-i\om^2\Ai'(\om z)=-i\om^2\Ai'(\om^{-2} z)\sim\frac{i\om^2(\om^{-2} z)^{1/4}}{2\sqrt\pi}e^{-\frac23(\om^{-2} z)^{3/2}}
\sim\frac{-iz^{1/4}}{2\sqrt\pi}e^{-\frac23 z^{3/2}},
\end{eqnarray*}
and
\begin{eqnarray*}
&&\om^2\Ai(\om^2z)=\om^2\Ai(\om^{-1}z)\sim\frac{\om^2(\om^{-1} z)^{-1/4}}{2\sqrt\pi}e^{-\frac23(\om^{-1} z)^{3/2}}
\sim\frac{-iz^{-1/4}}{2\sqrt\pi}e^{\frac23 z^{3/2}},\\
&&i\om\Ai'(\om^2 z)=i\om\Ai'(\om^{-1} z)\sim\frac{-i\om(\om^{-1} z)^{1/4}}{2\sqrt\pi}e^{-\frac23(\om^{-1} z)^{3/2}}
\sim\frac{z^{1/4}}{2\sqrt\pi}e^{\frac23 z^{3/2}}.
\end{eqnarray*}
For $\arg z\in[-\pi,-\pi/3)$, we have $\arg(\om^2 z)\in[\pi/3,\pi)$ and $\arg(\om z)\in[-\pi/3,\pi/3)$.
Thus, as $z\to\infty$ with $\arg z\in[-\pi,-\pi/3)$, we obtain from (\ref{Airy}) that
\begin{eqnarray*}
&&-\om^2\Ai(\om^2z)\sim\frac{-\om^2(\om^2 z)^{-1/4}}{2\sqrt\pi}e^{-\frac23(\om^2 z)^{3/2}}
\sim\frac{z^{-1/4}}{2\sqrt\pi}e^{-\frac23 z^{3/2}},\\
&&-i\om\Ai'(\om^2 z)\sim\frac{i\om(\om^2 z)^{1/4}}{2\sqrt\pi}e^{-\frac23(\om^2 z)^{3/2}}
\sim\frac{-iz^{1/4}}{2\sqrt\pi}e^{-\frac23 z^{3/2}},
\end{eqnarray*}
and
\begin{eqnarray*}
&&-\om\Ai(\om z)\sim\frac{-\om(\om z)^{-1/4}}{2\sqrt\pi}e^{-\frac23(\om z)^{3/2}}
\sim\frac{-iz^{-1/4}}{2\sqrt\pi}e^{\frac23 z^{3/2}},\\
&&-i\om^2\Ai'(\om z)\sim\frac{i\om^2(\om z)^{1/4}}{2\sqrt\pi}e^{-\frac23(\om z)^{3/2}}
\sim\frac{z^{1/4}}{2\sqrt\pi}e^{\frac23 z^{3/2}}.\\
\end{eqnarray*}
Applying the last eight formulas to (\ref{A2}) gives
\begin{eqnarray}\label{A-asymp2}
   \A(z)\left(\begin{matrix}
    1&0 \\
    \pm1&1
  \end{matrix}\right)=\frac {z^{-\s_3/4}}{2\sqrt\pi} \left(\begin{matrix}
    1&-i \\
    -i&1
  \end{matrix}\right)(I+O(|z|^{-3/2}))e^{-\frac 2 3z^{3/2}\s_3}
\end{eqnarray}
as $z\to\infty$ with $|\arg z|\in(\pi/3,\pi]$.
Here the sign $\pm$ means plus sign when $z$ is in the upper half plane,
and minus sign when $z$ is in the lower half plane.
Finally, we introduce a crucial function which
enables us to obtain global asymptotic formulas without any cut in the complex plane; see a statement in the second last paragraph of Section 1.
For $z$ not on the imaginary line, we define
\begin{eqnarray}\label{D}
D(z):=\begin{cases}
  \cfrac{e^{nz}\G(nz-\b/2+1)}{\sqrt{2\pi}(nz)^{nz+(1-\b)/2}}&\re z>0;\\
  \\
  \cfrac{\sqrt{2\pi}(-nz)^{-nz+(\b-1)/2}}{e^{-nz}\G(-nz+\b/2)}&\re z<0.
\end{cases}
\end{eqnarray}
The jump of $D(z)$ on the imaginary line is given by
\begin{eqnarray}\label{JoD}
J_D(z):=D_-(z)^{-1}D_+(z)=\mp2i\sin\pi(nz-\b/2)e^{\pm i\pi(nz-\b/2)}=1-e^{\pm2i\pi(nz-\b/2)},
\end{eqnarray}
where $D_+(z)$ ($D_-(z)$) is the limiting value of $D(z)$ taken from the left (right) of the imaginary axis.
It will be seen in the proof of Lemma \ref{lem-K} that
the usage of $D(z)$ is to cancel the jump $1-e^{\pm2i\pi(nz-\b/2)}$ across the imaginary axis.
The explicit formula of $D(z)$ is obtained by
solving a one-dimensional Riemann-Hilbert problem and calculating a Cauchy integral.
As $n\to\infty$, applying Stirling's formula (cf. \cite[(6.1.40)]{AS70}) to (\ref{D}), we have
\begin{eqnarray}\label{D-asymp}
D(z)=1+O(1/n)
\end{eqnarray}
uniformly for $z$ bounded away from the origin.

\section{Construction of parametrix}
\begin{subequations}\label{R'}
For $\re z\notin[0,1]$ or $\im z\notin[-\d,\d]$, we define
\begin{eqnarray}\label{R'a}
  \wt R(z):=\sqrt\pi[Ce^{nl}]^{\s_3/2}N(z)(z-1)^{\frac{\b-1}2\s_3}\left(\begin{matrix}
    1&i \\
    i&1
  \end{matrix}\right)F(z)^{\s_3/4}\A(F)
  %\nonumber\\  &&\times
  \left[\frac{D(z)^2e^{nv(z)}}{Cz^{\b-1}B(z)^2}\right]^{\s_3/2}.
\end{eqnarray}
For $\re z\in(0,1)$ and $\im z\in(-\delta,\delta)$, we define
\begin{eqnarray}\label{R'b}
  \wt R(z)&:=&(-1)^{n+1}\sqrt\pi[Ce^{nl}]^{\s_3/2}N(z)(1-z)^{\frac{\b-1}2\s_3}\left(\begin{matrix}
    1&i \\
    i&1
  \end{matrix}\right)\wt F(z)^{-\s_3/4}\s_1\A(\wt F)\s_1\nonumber\\
  &&\times\left[\frac{4\sin^2\th(z)D(z)^2e^{nv(z)}}{Cz^{\b-1}B(z)^2}\right]^{\s_3/2}.
\end{eqnarray}
\end{subequations}
Note that $\tilde{R}(z)$ has jumps across the negative real axis and the imaginary axis; they are caused by the functions $z^{\beta-1}$ and $D(z)$ respectively. Our parametrix $\tilde{R}(z)$ is analogous to that in \cite[(4.54)]{DW07}.
The main difference is that the parametrix here has a factor involving the auxiliary function $D(z)$ defined in (\ref{D}).
This factor will make our asymptotic formulas valid in much bigger regions, one of which includes the cut $(-\infty,0]$.
As $z\to\infty$, a combination of (\ref{phi-asymp}), (\ref{F}) and (\ref{A-asymp1}) gives
$$\sqrt\pi\left(\begin{matrix}
    1&i \\
    i&1
  \end{matrix}\right)F(z)^{\s_3/4}\A(F)= e^{-n\phi\s_3}[I+O(\frac{1}{z})]= e^{(-\frac{nv}2-\frac{nl}2)\s_3}z^{n\s_3}[I+O(\frac{1}{z})].$$
Furthermore, it is easily seen from the definitions (\ref{B}), (\ref{N}) and (\ref{D}) that we have respectively
$B(z)\sim z^n$, $N(z)\sim I$ and $D(z)\sim1$ as $z\to\infty$.
Thus we obtain from (\ref{R'a}) and the above formula that
\begin{eqnarray}\label{R'-asymp}
\wt R(z)=[Ce^{nl}]^{\s_3/2}z^{\frac{\b-1}2\s_3}[I+O(\frac{1}{z})]e^{(-\frac{nv}2-\frac{nl}2)\s_3}z^{n\s_3}
\left[\frac{e^{nv}}{Cz^{\b-1}z^{2n}}\right]^{\s_3/2}=I+O(\frac{1}{z})
\end{eqnarray}
as $z\to\infty$.
Define
\begin{eqnarray}\label{K}
K(z):=(Ce^{nl})^{-\s_3/2}R(z)\wt R(z)^{-1}(Ce^{nl})^{\s_3/2}.
\end{eqnarray}
It is clear from (\ref{R'-asymp}) and Proposition \ref{prop-R} that
\begin{eqnarray}\label{K-asymp}
K(z)=I+O(\frac{1}{z})
\end{eqnarray}
as $z\to\infty$.
Let $\Sigma_{K}$ denote the oriented contour consisting of $\Sigma_{R}$ in Figure {\textup 1}, the negative real axis, and the two infinite lines from $z=\pm i\delta$ to $z=\pm i \infty$ on the imaginary axis. The jump matrix of $K(z)$ is given by
\begin{eqnarray}\label{JoK}
J_K(z):=K_-(z)^{-1}K_+(z)=(Ce^{nl})^{-\s_3/2}\wt R_-(z)J_R(z)\wt R_+(z)^{-1}(Ce^{nl})^{\s_3/2}.
\end{eqnarray}

\begin{lem}\label{lem-K}
$J_K(z)=I+O(1/n)$ and $K(z)=I+O(1/n)$ as $n\to\infty$.
\end{lem}

%\begin{proof}
\noindent {\it Proof.} In view of the structure of the contour $\Sigma_{K}$, we divide our discussion into eight cases and consider each case separately.

Case I.
For ${\textup R{\textup e}}\,z=1$ and $\pm\im z\in[0,\d]$, we have from (\ref{JoR1}) and (\ref{naD}) that
$$J_R(z)=\left(\begin{matrix}
    1-e^{\pm2i\th}&\cfrac{Cz^{\b-1}B^2W}{2i\sin\th e^{\mp i\th+nv}}\\
    \\
    \cfrac{-2i\sin\th e^{\pm i\th+nv}}{Cz^{\b-1}B^2W}&1
  \end{matrix}\right).$$
This together with (\ref{R'}) and (\ref{JoK}) gives
\begin{eqnarray}\label{JoK1}
J_K(z)&=&N(z)(z-1)^{\frac{\b-1}2\s_3}
\left[\left(\begin{matrix}
    1&i \\
    i&1
  \end{matrix}\right)F(z)^{\s_3/4}\A(F)\right]
  \left(\begin{matrix}
    \cfrac{1-e^{\pm2i\th}}{2\sin\th}&-iD^2We^{\pm i\th}\\
    \\
    \cfrac{-ie^{\pm i\th}}{D^2W}&2\sin\th
  \end{matrix}\right)\nonumber\\
  &&\times \left[\left(\begin{matrix}
    1&i \\
    i&1
  \end{matrix}\right)\wt F(z)^{-\s_3/4}\s_1\A(\wt F)\s_1\right]^{-1}
  (1-z)^{\frac{1-\b}2\s_3}N(z)^{-1}(-1)^{n+1}.
\end{eqnarray}
On account of (\ref{F}), we obtain from (\ref{A-asymp2}) that as $n\rightarrow\infty$,
\begin{eqnarray*}
\left(\begin{matrix}
    1&i \\
    i&1
  \end{matrix}\right)F(z)^{\s_3/4}\A(F)=[I+O(\frac{1}{n})]\frac{e^{-n\phi\s_3}}{\sqrt\pi}\left(\begin{matrix}
    1&0 \\
    \mp1&1
  \end{matrix}\right)
\end{eqnarray*}
and
\begin{eqnarray*}
\left[\left(\begin{matrix}
    1&i \\
    i&1
  \end{matrix}\right)\wt F(z)^{-\s_3/4}\s_1\A(\wt F)\s_1\right]^{-1}
  =\sqrt\pi\left(\begin{matrix}
    1&\mp1   \\
    0&1
  \end{matrix}\right)e^{n\wt\phi\s_3}[I+O(\frac{1}{n})].
\end{eqnarray*}
From (\ref{theta}) and (\ref{phi-phi'}) we have
$$
e^{n\wt\phi}=(-1)^ne^{n\phi\mp i\th\mp i\pi\b/2}.
$$
As $n\to\infty$, applying Stirling's formula (cf. \cite[(6.1.40)]{AS70}) to (\ref{W}) yields
\begin{eqnarray}\label{W-asymp}
W(z)=1+O(1/n)
\end{eqnarray}
uniformly for $z$ bounded away from the negative real axis.
Applying the last four equations and (\ref{D-asymp}) to (\ref{JoK1}) gives
\begin{eqnarray*}
J_K(z)&=&N(z)(z-1)^{\frac{\b-1}2\s_3}
  \left(\begin{matrix}
    1&\pm(1-D^2W)e^{-2n\wt\phi\mp i\pi\b}\\
    \\
    \mp(1-D^{-2}W^{-1})e^{2n\phi}&1-e^{\pm2i\th}(2-D^2W-D^{-2}W^{-1})
  \end{matrix}\right)\nonumber\\
  &&\times(z-1)^{\frac{1-\b}2\s_3}N(z)^{-1}[I+O(\frac{1}{n})]
  \\&=&I+O(\frac{1}{n}).
\end{eqnarray*}
Here we have used the fact that in the present case, $\re\phi(z)\le0$ and $\re\wt\phi(z)\ge0$; see Figure \ref{fig-phi}.

Case II.
For $z=x\in[1,\infty)$, we have from (\ref{JoR2}) and (\ref{naD})
$$J_R(x)=\left(\begin{matrix}
    1&-Cx^{\b-1}B^2We^{-nv}
    \\
    0&1
  \end{matrix}\right).$$
This together with (\ref{JoA}), (\ref{R'}) and (\ref{JoK}) gives
\begin{eqnarray}\label{JoK2}
J_K(x)&=&\left[N(x)(x-1)^{\frac{\b-1}2\s_3}\left(\begin{matrix}
    1&i \\
    i&1
  \end{matrix}\right)F(x)^{\s_3/4}\A_-(F)\right]
  \left(\begin{matrix}
    1&1-D^2W
    \\
    0&1
  \end{matrix}\right)\nonumber\\
  &&\times \left[N(x)(x-1)^{\frac{\b-1}2\s_3}\left(\begin{matrix}
    1&i \\
    i&1
  \end{matrix}\right)F(x)^{\s_3/4}\A_-(F)\right]^{-1}.
\end{eqnarray}
Note that the matrices $N(z)$ and $F(z)^{\s_3/4}$ are both discontinuous across the interval $[1,b)$.
But a combination of them makes the jumps vanish.
Observe from Figure \ref{fig-phi} and (\ref{F}) that
$\arg\phi_\pm(x)=\pm3\pi/2$ and $\arg F_\pm(x)=\pm\pi$ for $x\in[1,b)$.
Thus, we have $F_+(x)^{\s_3/4}=F_-(x)^{\s_3/4}e^{i\pi\s_3/2}$.
It then follows from (\ref{JoN}) that the matrix
$$N(z)(z-1)^{\frac{\b-1}2\s_3}\left(\begin{matrix}
    1&i \\
    i&1
  \end{matrix}\right)F(z)^{\s_3/4}$$
has no jump on the interval $[1,b)$.
Applying (\ref{D-asymp}) and (\ref{W-asymp}) to (\ref{JoK2}) gives $J_K(x)=I+O(1/n)$.

Case III.
For $z=x\in[0,1]$, we can proceed in a similar manner as in Case II and obtain $J_K(x)=I+O(1/n)$.

Case IV.
For $z=\re z\pm i\d$ with $\re z\in(1,\infty)$, we have from (2.17b) and (\ref{naD})
$$J_R(z)=\left(\begin{matrix}
    1&\cfrac{Cz^{\b-1}B^2W}{2i\sin\th e^{\mp i\th+nv}}\\
    \\
    0&1
  \end{matrix}\right).$$
This together with (\ref{R'}) and (\ref{JoK}) gives
\begin{eqnarray}\label{JoK4}
J_K(z)&=&N(z)(z-1)^{\frac{\b-1}2\s_3}
\left[\left(\begin{matrix}
    1&i \\
    i&1
  \end{matrix}\right)F(z)^{\s_3/4}\A(F)e^{n\phi\s_3}\right]
  \left(\begin{matrix}
    1&\cfrac{D^2We^{-2n\phi\pm 2i\th}}{2i\sin\th e^{\pm i\th}}\\
    \\
    0&1
  \end{matrix}\right)\nonumber\\
  &&\times \left[\left(\begin{matrix}
    1&i \\
    i&1
  \end{matrix}\right)F(z)^{\s_3/4}\A(F)e^{n\phi\s_3}\right]^{-1}
  (z-1)^{\frac{1-\b}2\s_3}N(z)^{-1}.
\end{eqnarray}
Note from (\ref{phi1}) that in this case, we have $\arg F(z)\in (-\pi,\pi)$.
Thus, coupling (\ref{F}) and (\ref{A-asymp1}), we obtain
$$\left(\begin{matrix}
    1&i \\
    i&1
  \end{matrix}\right)F(z)^{\s_3/4}\A(F)e^{n\phi\s_3}=\frac1{\sqrt\pi}[I+O(\frac{1}{n})].$$
Applying this to (\ref{JoK4}) yields
\begin{eqnarray*}
J_K(z){=} I+O(\frac{1}{n}).
\end{eqnarray*}
Here we have used the facts that
$\mp2i\sin\th e^{\pm i\th}=1-e^{\pm2i\th}$ $\sim 1$ as $n\rightarrow\infty$
and $\re(-n\phi\pm i\th)=-n(\re\phi+\pi\d)<0$; see (\ref{theta}) and (\ref{phi3}).

Case V.
For $z=x \pm i\delta$ with $x \in(0,1)$, we have from (2.17a) and (\ref{naD})
$$J_R(z)=\left(\begin{matrix}
    1&0\\
    \\
    \cfrac{2i\sin\th e^{\pm i\th+nv}}{Cz^{\b-1}B^2W}&1
  \end{matrix}\right).$$
This together with (\ref{R'}) and (\ref{JoK}) gives
\begin{eqnarray}\label{JoK5}
J_K(z)& =&(-1)^{n+1}N(z)(1-z)^{\frac{\b-1}2\s_3}\left[\left(\begin{matrix}
    1&i \\
    i&1
  \end{matrix}\right)\wt F(z)^{-\s_3/4}\s_1\A(\wt F)\s_1\right]
  \left(\begin{matrix}
    2\sin\th&0
    \\
    iD^{-2}W^{-1}e^{\pm i\th}&(2\sin\th)^{-1}
  \end{matrix}\right)\nonumber\\
  &&\times \left[\left(\begin{matrix}
    1&i \\
    i&1
  \end{matrix}\right)F(z)^{\s_3/4}\A(F)\right]^{-1}
  (z-1)^{\frac{1-\b}2\s_3}N(z)^{-1}.
\end{eqnarray}
Note from (\ref{phi1}) that in this case, $\arg F(z)\in (-\pi,\pi)$ and $\arg\wt F(z)\in(-\pi,\pi)$. Thus,
we obtain from (\ref{F}) and (\ref{A-asymp1})
\begin{eqnarray*}
\left(\begin{matrix}
    1&i \\
    i&1
  \end{matrix}\right)\wt F(z)^{-\s_3/4}\s_1\A(\wt F)\s_1
  =[I+O(\frac{1}{n})]\frac{e^{-n\wt\phi\s_3}}{\sqrt\pi}
\end{eqnarray*}
and
\begin{eqnarray*}
\left[\left(\begin{matrix}
    1&i \\
    i&1
  \end{matrix}\right)F(z)^{\s_3/4}\A(F)\right]^{-1}=\sqrt\pi e^{n\phi\s_3}[I+O(\frac{1}{n})]
\end{eqnarray*}
Applying the last two equations to (\ref{JoK5}) yields
\begin{eqnarray*}
J_K(z)&=&N(z)(z-1)^{\frac{\b-1}2\s_3}
  \left(\begin{matrix}
    2(-1)^{n+1}\sin\th e^{n\phi-n\wt\phi\pm i\pi(1-\b)/2}&0\\
    \\
    i(-1)^{n+1}D^{-2}W^{-1}e^{n\phi+n\wt\phi\pm i\th\mp i\pi(1-\b)/2}&\cfrac{e^{n\wt\phi-n\phi\mp i\pi(1-\b)/2}}{2(-1)^{n+1}\sin\th}
  \end{matrix}\right)\nonumber\\
  &&\times
  (z-1)^{\frac{1-\b}2\s_3}N(z)^{-1}[I+O(\frac{1}{n})].
\end{eqnarray*}
Using (2.18), (3.2), (3.3) and (3.10), one can show that $$\re\{n\phi+n\wt\phi\pm i\th\}=\re\{2n\phi\}<0$$ and
$$2(-1)^{n+1}\sin\th e^{n\phi-n\wt\phi\pm i\pi(1-\b)/2}=\mp2i\sin\th e^{\pm i\th}=1-e^{\pm 2i\th}\sim1.$$
Thus, we again have
\begin{eqnarray*}
J_{k}(z)=I+O(\frac{1}{n}),\qquad\qquad {\textup a}{\textup s} \quad n\rightarrow\infty.
\end{eqnarray*}

Case VI.
For $z=\pm iy$ with $y \in(0,\d),$ we have from (\ref{JoR3}) and (\ref{naD})
$$J_R(z)=\left(\begin{matrix}
    1&0\\
    \\
    \cfrac{2i\sin\th e^{\pm i\th+nv}}{Cz^{\b-1}B^2W}&1
  \end{matrix}\right).$$
This together with (\ref{R'}) and (\ref{JoK}) gives
\begin{eqnarray}\label{JoK6}
J_K(z)&=&(-1)^{n+1}N(z)(1-z)^{\frac{\b-1}2\s_3}
\left[\left(\begin{matrix}
    1&i \\
    i&1
  \end{matrix}\right)\wt F(z)^{-\s_3/4}\s_1\A(\wt F)\s_1\right]
  \left(\begin{matrix}
    \cfrac{2D_-\sin\th}{D_+}&0\\
    \\
    \cfrac{ie^{\pm i\th}}{D_+D_-W}&\cfrac{D_+}{2D_-\sin\th}
  \end{matrix}\right)\nonumber\\
  &&\times \left[\left(\begin{matrix}
    1&i \\
    i&1
  \end{matrix}\right)F(z)^{\s_3/4}\A(F)\right]^{-1}
  (z-1)^{\frac{1-\b}2\s_3}N(z)^{-1}.
\end{eqnarray}
Note from (\ref{phi1}) that in this case, $\arg F(z)\in (-\pi,\pi)$ and $\arg\wt F(z)\in(-\pi,\pi)$. Thus, as in Case IV we have from (\ref{F}) and (\ref{A-asymp1})
\begin{eqnarray*}
\left(\begin{matrix}
    1&i \\
    i&1
  \end{matrix}\right)\wt F(z)^{-\s_3/4}\s_1\A(\wt F)\s_1
  =[I+O(\frac{1}{n})]\frac{e^{-n\wt\phi\s_3}}{\sqrt\pi}
\end{eqnarray*}
and
\begin{eqnarray*}
\left[\left(\begin{matrix}
    1&i \\
    i&1
  \end{matrix}\right)F(z)^{\s_3/4}\A(F)\right]^{-1}=\sqrt\pi e^{n\phi\s_3}[I+O(\frac{1}{n})].
\end{eqnarray*}
Applying the above two equations to (\ref{JoK6}) yields
\begin{eqnarray*}
J_K(z)&=&N(z)(z-1)^{\frac{\b-1}2\s_3}
  \left(\begin{matrix}
    \cfrac{2(-1)^{n+1}\sin\th}{J_De^{n\wt\phi-n\phi\mp i\pi(1-\b)/2}}&0\\
    \\
    O(e^{2n\re\phi})&\cfrac{J_De^{n\wt\phi-n\phi\mp i\pi(1-\b)/2}}{2(-1)^{n+1}\sin\th}
  \end{matrix}\right)\nonumber\\
  &&\times
  (z-1)^{\frac{1-\b}2\s_3}N(z)^{-1}[I+O(\frac{1}{n})].
\end{eqnarray*}
Here we have used (2.21), (3.20) and the asymptotic formula for $\Gamma(x\pm iy)$ as $y\rightarrow +\infty$. Since
$$2(-1)^{n+1}\sin\th e^{n\phi-n\wt\phi\pm i\pi(1-\b)/2}=1-e^{\pm 2i\th}=J_D$$
and $\re\phi(z)<0$, as before we again have $J_{K}(z)=I+O(1/n)$ as $n\rightarrow \infty$.
As mentioned in a statement following (3.21), the usage of $D(z)$ defined in (\ref{D}) is to cancel the jump
$1-e^{\pm 2i\th}=1-e^{\pm2i\pi(nz-\b/2)}$.
Without this function, the jump matrix $J_K(z)$ is not asymptotically equal to the identity matrix in this case.

Case VII.
For $\re z=0$ and $|\im z|\ge\d$, we have $J_R(z)=I$; see Figure 1. Thus, (\ref{R'}) and (\ref{JoK}) imply
\begin{eqnarray}\label{JoK7}
J_K(z)&=&\left[N(z)(z-1)^{\frac{\b-1}2\s_3}
\left(\begin{matrix}
    1&i \\
    i&1
  \end{matrix}\right)F(z)^{\s_3/4}\A(F)\right]
  J_D^{-\s_3}\nonumber\\
  &&\times
  \left[N(z)(z-1)^{\frac{\b-1}2\s_3}
\left(\begin{matrix}
    1&i \\
    i&1
  \end{matrix}\right)F(z)^{\s_3/4}\A(F)\right]^{-1}.
\end{eqnarray}
Note that by (\ref{JoD}), $J_D=1-e^{\pm2i\pi(nz-\b/2)}$ is exponentially small for $|\im z|\ge\d$.
From (\ref{JoK7}), it again follows that $J_K=I+O(1/n)$ in this case.

Case VIII.
For $z=x\in(-\infty,0)$, we have $J_R(x)=I$. Thus, (\ref{R'}) and (\ref{JoK}) imply
\begin{eqnarray}\label{JoK8}
J_K(z)&=&N(x)(1-x)^{\frac{\b-1}2\s_3}e^{\frac{i\pi(1-\b)}2\s_3}
\left[\left(\begin{matrix}
    1&i \\
    i&1
  \end{matrix}\right)F_-(x)^{\s_3/4}\A(F_-)e^{n\phi_-\s_3}\right]
  e^{i\pi(\b-1)\s_3}\nonumber\\
  &&\times
  \left[\left(\begin{matrix}
    1&i \\
    i&1
  \end{matrix}\right)F_+(x)^{\s_3/4}\A(F_+)e^{n\phi_+\s_3}\right]^{-1}
  (1-x)^{\frac{1-\b}2\s_3}e^{\frac{i\pi(1-\b)}2\s_3}N(x)^{-1}.
\end{eqnarray}
Here we have used the fact that $e^{n(\phi_+-\phi_-)}=e^{2ni\pi}=1$ for $x<0$; see (\ref{phi}).
Note from (\ref{phi1}) that $\arg F_\pm\in(-\pi,\pi)$ in this case.
Hence, we obtain from (\ref{F}) and (\ref{A-asymp1})
$$\left(\begin{matrix}
    1&i \\
    i&1
  \end{matrix}\right)F_\pm(x)^{\s_3/4}\A(F_\pm)e^{n\phi_\pm\s_3}
  =\frac1{\sqrt\pi}[I+O(\frac{1}{n})].$$
Applying the last equation to (\ref{JoK8}) yields
\begin{eqnarray*}
J_K(z)= N(x)(1-x)^{\frac{\b-1}2\s_3}e^{\frac{i\pi(1-\b)}2\s_3}
  e^{i\pi(\b-1)\s_3}
  e^{\frac{i\pi(1-\b)}2\s_3}(1-x)^{\frac{1-\b}2\s_3}N(x)^{-1}[I+O(\frac{1}{n})]=I+O(\frac{1}{n}).
\end{eqnarray*}

In conclusion, we have shown that $J_K(z)=I+O(1/n)$ on the contour of $K(z)$.
It is not difficult to verify that the multiplicative cyclic condition (3.30) in [14] holds for the jump matrix $J_K(z)$.
An application of Theorem 3.8 in [14] then gives $K(z)=I+O(1/n)$ as $n\to\infty$.\qquad $\blacksquare$
%\end{proof}

\begin{subequations}
Combining Lemma \ref{lem-K} with (\ref{P}), (\ref{B}), (\ref{Q}), (\ref{R}) and (\ref{K}), we obtain
\begin{eqnarray}\label{P-asymp1}
\pi_n(nz-\b/2)= n^nB(z)\wt R_{11}(z)[I+O(\frac{1}{n})]
\end{eqnarray}
for $\re z\notin[0,1]$ or $\im z\notin[-\d,\d]$, and
\begin{eqnarray}\label{P-asymp2}
\pi_n(nz-\b/2)= n^nB(z)[\wt R_{11}(z)+\D_\pm(z)\wt R_{12}(z)][I+O(\frac{1}{n})]
\end{eqnarray}
for $\re z\in(0,1)$ and $\im z\in(0,\pm\d)$.
\end{subequations}

\section{Main results}
\begin{thm}
As $n\to\infty$, we have
\begin{eqnarray}\label{pi-asymp1}
\pi_n(nz-\b/2)&=&n^n\sqrt\pi D(z)e^{nv(z)/2+nl/2}
    \bigg\{
  \frac{(\frac{\sqrt{z-a}+\sqrt{z-b}}2)^\b+(\frac{\sqrt{z-a}-\sqrt{z-b}}2)^\b}{z^{(\b-1)/2}(z-a)^{1/4}(z-b)^{1/4} F(z)^{-1/4}}\Ai(F)\nonumber\\
  &&-  \frac{(\frac{\sqrt{z-a}+\sqrt{z-b}}2)^\b-(\frac{\sqrt{z-a}-\sqrt{z-b}}2)^\b}{z^{(\b-1)/2}(z-a)^{1/4}(z-b)^{1/4} F(z)^{1/4}}\Ai'(F)\bigg\}[1+O(\frac{1}{n})]
\end{eqnarray}
for $\re z\notin[0,1]$ or $\im z\notin[-\d,\d]$, and
\begin{eqnarray}\label{pi-asymp2}
  \pi_n(nz-\b/2)&=&(-n)^n\sqrt\pi D(z)e^{nv(z)/2+nl/2}
    \nonumber\\
  &&\times\bigg\{
  \frac{(\frac{\sqrt{b-z}+\sqrt{a-z}}2)^\b+(\frac{\sqrt{b-z}-\sqrt{a-z}}2)^\b}{z^{(\b-1)/2}(b-z)^{1/4}(a-z)^{1/4}\wt F(z)^{-1/4}}
[\cos\th\Ai(\wt F)-\sin\th\Bi(\wt F)]
 \\
  &&\ \text{\,} +  \frac{(\frac{\sqrt{b-z}+\sqrt{a-z}}2)^\b-(\frac{\sqrt{b-z}-\sqrt{a-z}}2)^\b}{z^{(\b-1)/2}(b-z)^{1/4}(a-z)^{1/4}\wt F(z)^{1/4}}
  [\cos\th\Ai'(\wt F)-\sin\th\Bi'(\wt F)]\bigg\}[I+O(\frac{1}{n})] \nonumber
\end{eqnarray}
for $\re z\in(0,1)$ and $\im z\in(-\delta,\delta)$, where the constant $l$ is given in (3.4) and the functions $v(z), F(z)$ and $D(z)$ are respectively given in (2.19), (3.6) and (3.20).
The asymptotic formula on the boundary of the two regions can be obtained by taking limits from either side.
\end{thm}

%\begin{proof}
\noindent {\it Proof.} From (\ref{N}), (\ref{A}), (\ref{R'}) and (\ref{P-asymp1}), it is easy to obtain (\ref{pi-asymp1}).
We now prove (\ref{pi-asymp2}).
Define
\begin{eqnarray}\label{Q'}
\wt Q(z):=\wt R(z)\left(\begin{matrix}
    1&0 \\
    \D_\pm(z)&1
  \end{matrix}\right).
\end{eqnarray}
From (\ref{P-asymp2}), we have
\begin{eqnarray}\label{pi-Q'}
\pi_n(nz-\b/2)= n^nB(z)\wt Q_{11}(z)[I+O(\frac{1}{n})].
\end{eqnarray}
Thus, we only need to calculate $\wt Q_{11}(z)$.
First, we observe from (\ref{naD}), (\ref{R'}) and (\ref{Q'}) that
\begin{eqnarray*}
\wt Q(z)&=&(-1)^{n+1}\sqrt\pi[Ce^{nl}]^{\s_3/2}N(z)(1-z)^{\frac{\b-1}2\s_3}\left(\begin{matrix}
    1&i \\
    i&1
  \end{matrix}\right)\wt F(z)^{-\s_3/4}\s_1\A(\wt F)\s_1
  \nonumber\\&&\times
  \left(\begin{matrix}
    1&0 \\
    \\
    \cfrac{ie^{\pm i\th(z)}}{2D^2(z)W(z)\sin\th(z)}&1
  \end{matrix}\right)
  \left[\frac{4\sin^2\th(z)D(z)^2e^{nv(z)}}{Cz^{\b-1}B(z)^2}\right]^{\s_3/2}.
\end{eqnarray*}
Second, (\ref{N}) gives
\begin{eqnarray*}
&&N(z)(1-z)^{\frac{\b-1}2\s_3}\left(\begin{matrix}
    1&i \\
    i&1
  \end{matrix}\right)
  \\&=&
  \left(\begin{matrix}
    \cfrac{(\frac{\sqrt{b-z}+\sqrt{a-z}}2)^\b-(\frac{\sqrt{b-z}-\sqrt{a-z}}2)^\b}{(b-z)^{1/4}(a-z)^{1/4}}&
    \cfrac{(\frac{\sqrt{b-z}+\sqrt{a-z}}2)^\b+(\frac{\sqrt{b-z}-\sqrt{a-z}}2)^\b}{-i(b-z)^{1/4}(a-z)^{1/4}}\\
    \\
    \cfrac{(\frac{\sqrt{b-z}+\sqrt{a-z}}2)^{2-\b}-(\frac{\sqrt{b-z}-\sqrt{a-z}}2)^{2-\b}}{-i(b-z)^{1/4}(a-z)^{1/4}}&
    \cfrac{(\frac{\sqrt{b-z}+\sqrt{a-z}}2)^{2-\b}+(\frac{\sqrt{b-z}-\sqrt{a-z}}2)^{2-\b}}{(b-z)^{1/4}(a-z)^{1/4}}
  \end{matrix}\right).
\end{eqnarray*}
Finally, (\ref{A}) implies
\begin{eqnarray*}
&&\s_1\A(\wt F)\s_1\left(\begin{matrix}
    1&0 \\
    \\
    \cfrac{ie^{\pm i\th}}{2D^2W\sin\th}&1
  \end{matrix}\right)(2\sin\th)^{\s_3}
  \\&=&
  \left(\begin{matrix}
    \Bi'(\wt F)&i\Ai'(\wt F)
    \\
    -i\Bi(\wt F)&\Ai(\wt F)
  \end{matrix}\right)\left(\begin{matrix}
    1/2&0
    \\
    \pm1/2&1
  \end{matrix}\right)\left(\begin{matrix}
    2\sin\th&0
    \\
    iD^{-2}W^{-1}e^{\pm i\th}&(2\sin\th)^{-1}
  \end{matrix}\right)
  \\&=&\left(\begin{matrix}
    \sin\th\Bi'(\wt F)-[\cos\th+(D^{-2}W^{-1}-1)e^{\pm i\th}]\Ai'(\wt F)&\cfrac{i\Ai'(\wt F)}{2\sin\th} \\
    \\
    -i\{\sin\th\Bi(\wt F)-[\cos\th+(D^{-2}W^{-1}-1)e^{\pm i\th}]\Ai(\wt F)\}&\cfrac{\Ai(\wt F)}{2\sin\th}
  \end{matrix}\right).
\end{eqnarray*}
Applying the last three equations to (\ref{pi-Q'}) gives
\begin{eqnarray*}
  \pi_n(nz-\b/2)&=&(-n)^n\sqrt\pi D(z)e^{nv(z)/2+nl/2}
  \bigg\{
  \frac{(\frac{\sqrt{b-z}+\sqrt{a-z}}2)^\b+(\frac{\sqrt{b-z}-\sqrt{a-z}}2)^\b}{z^{(\b-1)/2}(b-z)^{1/4}(a-z)^{1/4}\wt F(z)^{-1/4}}
  \nonumber\\
  &&\times[\cos\th\Ai(\wt F)-\sin\th\Bi(\wt F)+e^{\pm i\th}(D^{-2}W^{-1}-1)\Ai(\wt F)]
  \nonumber\\
  &&+  \frac{(\frac{\sqrt{b-z}+\sqrt{a-z}}2)^\b-(\frac{\sqrt{b-z}-\sqrt{a-z}}2)^\b}{z^{(\b-1)/2}(b-z)^{1/4}(a-z)^{1/4}\wt F(z)^{1/4}}
  \nonumber\\
  &&\times[\cos\th\Ai'(\wt F)-\sin\th\Bi'(\wt F)+e^{\pm i\th}(D^{-2}W^{-1}-1)\Ai'(\wt F)]\bigg\}[I+O(\frac{1}{n})].
\end{eqnarray*}
On account of (\ref{phi}), (\ref{F}), and (\ref{Airy}), $\Ai(\wt F)$ and $\Ai'(\wt F)$ are exponentially small when $z$ approaches the origin; by (\ref{D-asymp}) and (\ref{W-asymp}), we also have
$D^{-2}W^{-1}-1=O(1/n)$ for $z\neq0$.
Since we can always neglect the terms $(D^{-2}W^{-1}-1)\Ai(\wt F)$ and $(D^{-2}W^{-1}-1)\Ai'(\wt F)$,
formula (5.2) is proved.
%\end{proof}

To justify that the asymptotic formula on the curve separating the two regions can be obtain by taking limits from either side, we just note that the regions of validity of both formulas (5.1) and (5.2) can be slightly extended beyond their boundaries, and that in the overlapping region these two formulas are asymptotically equal. \qquad $\blacksquare$

\begin{rem}
We would like to mention that our results coincide with those obtained in \cite{JW98,JW99}.
The formulas (6.9) in [8] and (2.35) in [9] are asymptotically equal to (\ref{pi-asymp1}) {\textup {in the present paper}},
while the formulas (6.27) in [8] and (4.19) in [9] are asymptotically equal to (\ref{pi-asymp2}).
\end{rem}

\nocite{*}


\begin{thebibliography}{99}

\bibitem{AS70} M. Abramowitz, I. A. Stegun, \textsl{``Handbook of Mathematical Functions,
with Formulas, Graphs, and Mathematical Tables"}, Dover Publications, Inc., New York, 1970.

\bibitem{BKMM03} J. Baik, T. Kriecherbauer, K. T.-R. McLaughlin,
P. D. Miller, Uniform asymptotics for polynomials
orthogonal with respect to a general class of discrete weights and
universality results for associated ensembles: announcement of
results,  {\it Int. Math. Res. Not.} 2003, no. 15, 821--858.

\bibitem{BKMM07} J. Baik, T. Kriecherbauer, K. T.-R. McLaughlin,
P. D. Miller, \textsl{``Discrete Orthogonal Polynomials. Asymptotics
and Applications"}, Annals of Mathematics Studies, 164. Princeton
University Press, Princeton, NJ, 2007. viii+170 pp.

\bibitem{DW07} D. Dai, R. Wong, Global asymptotics of Krawtchouk
polynomials - a Riemann-Hilbert approach, {\it Chin. Ann. Math. Ser. B} {\bf 28} (2007), 1-34.

\bibitem{DKMVZ99} P. Deift, T. Kriecherbauer, K. T.-R. McLaughlin,
S. Venakides, and X. Zhou, Strong asymptotics of
orthogonal polynomials with respect to exponential weights, {\it Comm.
Pure Appl. Math.} {\bf 52} (1999), 1491-1552.

\bibitem{DZ93} P. Deift, X. Zhou, A steepest descent method for
oscillatory Riemann-Hilbert problems. Asymptotic for the MKdV equation,
{\it Ann. of Math.} {\bf 137} (1993), 295-368.

\bibitem{EMOT53} A. Erd\'elyi, W. Magnus, F. Oberhettinger, F. G. Tricomi,
{\it Higher transcendental functions. Vols. I {\textup {\&}} II,}
McGraw-Hill Book Company, Inc., New York, 1953.

\bibitem{JW98} X.-S. Jin, R. Wong, Uniform asymptotic
expansions for Meixner polynomials, {\it Constr. Approx.} {\bf 14} (1998),
113-150.

\bibitem{JW99} X.-S. Jin, R. Wong, Asymptotic formulas
for the zeros of the Meixner polynomials,  {\it J. Approx. Theory} {\bf 96}
(1999), 281-300.

\bibitem{Jo00} K. Johansson, Shape fluctuations and random matrices,
{\it Comm. Math. Phys.} {\bf 209} (2000), 437-476.

\bibitem{MV85} M. Maejima, W. Van Assche, Probabilistic
proofs of asymptotic formulas for some classical polynomial,
{\it Math. Proc. Cambridge Philos. Soc.} {\bf 97} (1985), 499-510.

\bibitem{Ol97} F. W. J. Olver, \textsl{``Asymptotics and Special
Functions"}, Academic Press, New York, 1974. Reprinted by A. K.
Peters, Wellesley, MA, 1997.

\bibitem{OW10} C. Ou, R. Wong,
The Riemann-Hilbert approach to global asymptotics of discrete orthogonal polynomials with infinite nodes,
{\it Anal. Appl.} {\bf 8} (2010), 247-286.

\bibitem{QW08} W.-Y. Qiu, R. Wong,
Asymptotic expansions for Riemann-Hilbert problems,
{\it Anal. Appl.} {\bf 6} (2008), 269-298.

\bibitem{WW05} Z. Wang, R. Wong,
Uniform asymptotics for orthogonal polynomials with exponential weights--the Riemann-Hilbert approach,
{\it Stud. Appl. Math.} {\bf 115} (2005), 139-155.

\bibitem{Wo89} R. Wong, \textsl{``Asymptotic Approximations of
Integrals"}, Academic Press, Boston, 1989. Reprinted by SIAM,
Philadelphia, PA, 2001.

\bibitem{WZ07} R. Wong, L. Zhang,
Global asymptotics of Hermite polynomials via Riemann-Hilbert approach,
{\it Discrete Contin. Dyn. Syst. Ser. B} {\bf 7} (2007), 661-682.

\bibitem{WZ06} R. Wong, W.-J. Zhang,
Uniform asymptotics for Jacobi polynomials with varying large negative parameters--a Riemann-Hilbert approach,
{\it Trans. Amer. Math. Soc.} {\bf 358} (2006), 2663-2694.






\end{thebibliography}
\end{document}